\documentclass[11pt]{article}
\usepackage{amsmath,amssymb,mathrsfs,amsthm}

\newcommand{\Q}{\mathbb{Q}}
\newcommand{\N}{\mathbb{N}}
\newcommand{\R}{\mathbb{R}}
\newcommand{\Z}{\mathbb{Z}}
\newcommand{\e}{\epsilon}
\renewcommand{\S}{\mathbb{S}}
\renewcommand{\P}{\mathbb{P}}
\newcommand{\E}{\mathbb{E}}
\newcommand{\D}{\mathcal{D}}
\newcommand{\clD}{\overline{\mathcal{D}}}
\newcommand{\B}{\mathcal{B}}
\newcommand{\A}{\mathcal{A}}
\newcommand{\T}{\mathcal{T}}
\renewcommand{\O}{\mathcal{O}}
\newcommand{\U}{\mathscr{U}}
\newcommand{\V}{\mathscr{V}}
\newcommand{\set}[1]{\left\{#1\right\}}

{\bf}{\it}
{\bf}{\it}
{\bf}{\it}
{\bf}{\it}
{\bf}{\it}
\newtheorem{thm}{Theorem}
\newtheorem{lmm}{Lemma}
\newtheorem{prp}{Proposition}

\title{Replica overlap and covering time for the Wiener sausages among Poissonian obstacles}
\author{Ryoki Fukushima\footnote{
              Division of Mathematics, Graduate School of Science, 
              Kyoto University} 
}

\begin{document}

\maketitle


\begin{abstract}
We study two objects concerning the Wiener sausage among Poissonian obstacles. 
The first is the asymptotics for the \textit{replica overlap}, which is the intersection of two independent Wiener sausages. 
We show that it is asymptotically equal to their union. 
This result confirms that the localizing effect of the media is so strong as to completely determine the motional range of particles. 
The second is an estimate on the \textit{covering time}. It is known that the Wiener sausage avoiding Poissonian obstacles 
up to time $t$ is confined in some `clearing' ball near the origin and almost fills it. 
We prove here that the time needed to fill the confinement ball has the same order as its volume. 
\end{abstract}

\section*{Acknowledgement} 
I would like to thank Professor Nobuo Yoshida for a lot of helpful discussions and careful reading 
of the early version of the manuscript.

\section{Introduction}
We study two objects concerning the Wiener sausage among Poissonian obstacles. 
The first is the asymptotics for the \textit{replica overlap}, which is the intersection of two independent Wiener sausages. 
As in the mean field theory of spin glasses, it measures the strength of disorder. 
We show that it has asymptotically the same volume as the union of two Wiener sausages. 
Therefore, two independent Wiener sausages form the same shape, conditioned to avoid Poissonian obstacles. 
This result confirms that the localizing effect of the media is so strong as to completely determine 
the motional range of particles. 
The second is an estimate on the \textit{covering time}. It is known that the Brownian motion avoiding Poissonian obstacles 
up to time $t$ is typically confined in some `clearing' ball near the origin (see \cite{Szn91b}, \cite{Pov99}) 
and the author has shown in \cite{F1} that the corresponding Wiener sausage almost fills the ball. 
Since the volume of the confinement ball is smaller than the typical volume of the unconditional Wiener sausage 
at time $t$, it is natural to expect that the covering time is shorter than $t$. 
We prove this by showing that the covering time has the same order as the the volume of the confinement ball.  

\subsection{The model}
Let $(\Omega, \P_{\nu})$ be the Poisson point process of constant intensity $\nu$ on $\R^d$. 
We define the hard obstacles $S(\omega)=\bigcup_i (x_i+K)$ for a fixed nonpolar compact subset $K$ of $\R^d$ 
and $\Omega \ni \omega = \sum_i \delta_{x_i}$. Similarly, we define the soft obstacles 
$V(x,\omega)=\sum_i W(x-x_i)$ for a nonnegative, compactly supported and bounded measurable function $W$ 
which is not identically zero and $\Omega \ni \omega = \sum_i \delta_{x_i}$. 
Next, $((Z_t)_{t \ge 0}, P_x)$ denotes the standard Brownian motion starting from $x \in \R^d$. 
For an open set $U \subset \R^d$ and a closed set $F \subset \R^d$, $T_U= \inf \set{s \ge 0 \,;
\, Z_s \notin U}$ and $H_F = \inf \set{s \ge 0 \,;\, Z_s \in F}$ are the exit time 
of $U$ and the entrance time of $F$, respectively.

We define the annealed path measure for one particle by
\begin{equation*}
Q_t^{\mu, \nu}= \frac{1}{S_t^{\mu, \nu}}\exp \left\{-\int_0^t V(Z_s,\omega_1)ds \right\} 
1_{\{H_{S(\omega_2)}>t\}}\P_{\mu}^1 \otimes \P_{\nu}^2 \otimes P_0\label{1}
\end{equation*}
on $\Omega^2 \times C([0,t],\R^d)$ with $S_t^{\mu, \nu}$ the normalizing constant. 
Similarly, we also define the annealed path measure for two particles by 
\begin{equation*}
 \begin{split}
  \Q_t^{\mu, \nu}= \frac{1}{\S_t^{\mu, \nu}} &\exp \left\{-\int_0^t V(Z_s,\omega_1)ds-
  \int_0^t V(\Hat{Z}_s,\omega_1)ds \right\}\\
  &\hspace{60.5pt}1_{\{H_{S(\omega_2)}>t, \Hat{H}_{S(\omega_2)}>t\}} 
  \P_{\mu}^1 \otimes \P_{\nu}^2 \otimes P_0 \otimes \Hat{P}_0 \label{2}
 \end{split}
\end{equation*}
on $\Omega^2 \times C([0,t],\R^d)^2$ with $\S_t^{\mu, \nu}$ the normalizing constant. 
Finally, we introduce the Wiener sausage $W_t^C = \bigcup_{0 \le s \le t}(Z_s+C)$ associated with 
a compact set $C$. 

\subsection{Main results}
The first result in this article is that the intersection of two independent 
Wiener sausages is asymptotically equal to their union. 

\begin{thm}\label{Theorem1}
Let $d \ge 2$. Then for any $\eta>0$ and nonpolar compact set $C \subset \R^d$, 
 \begin{gather*}
  \lim_{t \to \infty}\Q_t^{\mu,\nu}\left(\left|t^{-\frac{d}{d+2}}|W_t^C \cup \Hat{W}_t^C|-
  2^{\frac{d}{d+2}}\omega_d R_0(d,\mu+\nu)^d \right|>\eta \right)=0,\label{Th1-1} \\
  \lim_{t \to \infty}\Q_t^{\mu,\nu}\left(\left|t^{-\frac{d}{d+2}}|W_t^C \cap \Hat{W}_t^C|-
  2^{\frac{d}{d+2}}\omega_d R_0(d,\mu+\nu)^d \right|>\eta \right)=0.\label{Th1-2}
 \end{gather*} 
 Here $R_0(d,\mu+\nu)>0$ is the radius of the ball 
which achieves the infimum $c(d,\mu+\nu)$ of the variational problem $\inf_{U:{\rm open}}\{(\mu+\nu)|U|+\lambda(U) \}$ 
with $\lambda(U)$ the principal eigenvalue of $-1/2\Delta$ on $H_0^1(U)$. 
\end{thm}
\vspace{11pt}\noindent
Theorem 1 gives the asymptotics for the volume but we are also able to determine the shape. 
It is a consequence of the next theorem. 

\begin{thm}\label{Theorem2}(Confinement property of two particles)\\
Let $d \ge 2$. There exist constants $\kappa_1 >1$ and $0<\kappa_2<1$ and for each 
$(\omega_1,\omega_2) \in \Omega^2$ a ball $B(\omega_1,\omega_2)$ with center in 
$B(0, 2^{1/(d+2)}R_0(d,\mu+\nu)+\kappa_1 t^{-\kappa_2/(d+2)})$ and radius in 
$[2^{1/(d+2)}R_0(d,\mu+\nu),2^{1/(d+2)}R_0(d,\mu+\nu)+\kappa_1 t^{-\kappa_2/(d+2)}]$ 
such that
\begin{equation*}
 \lim_{t \to \infty}\Q_t^{\mu,\nu}\left(T_{t^{1/(d+2)}B(\omega_1,\omega_2)}>t,
 \Hat{T}_{t^{1/(d+2)}B(\omega_1,\omega_2)}>t\right)=1.\label{5}
\end{equation*}
\end{thm}
\vspace{11pt}\noindent
Actually, combining Theorem 1 and Theorem 2, we have 
\begin{equation*}
 W_t^C \cup \Hat{W}_t^C, W_t^C \cap \Hat{W}_t^C \sim t^{\frac{1}{d+2}}B(\omega_1,\omega_2)\label{6}
\end{equation*}
in `measurable sense', that is, the symmetrical differences have small volumes. 

The second object is an estimate on the covering time of the confinement ball by the single Wiener sausage. 
If one considers the Wiener sausage conditioned to stay in the ball of radius $R_0(d,\nu)t^{1/(d+2)}$, 
it takes not longer than $t^{d/(d+2)}$ to cover almost all the area of the ball. 
This can be proved by the same argument as to prove Proposition 3.2.7 in \cite{Szn98}.   
In view of the confinement property (see \cite{F1} and references therein), we expect 
that the same estimate holds for our model. 
The following result gives an answer in the special case $C^{\circ} \neq \emptyset$. 
\begin{thm}\label{Theorem3}
Let $d \ge 2$ and assume that $C$ has nonempty interior. 
Then for any $d/(d+2) < \sigma \le 1$ and $\eta>0$, 
\begin{equation}
 \lim_{t \to \infty}Q_t^{\mu,\nu}\left(\left|t^{-\frac{d}{d+2}}|W_{t^{\sigma}}^C|-\omega_d 
 R_0(d,\mu+\nu)^d \right|>\eta \right)=0.\label{7}
\end{equation}
Conversely, for any $\sigma<d/(d+2)$ and $\eta>0$,
\begin{equation}
 \lim_{t \to \infty}Q_t^{\mu,\nu}\left(\left|t^{-\frac{d}{d+2}}|W_{t^{\sigma}}^C| \right|>\eta \right)=0.\label{8}
\end{equation}
\end{thm}
\vspace{11pt}\noindent
{\bf Remark.} In two dimensional case, it has been shown in \cite{Szn91b} that there are no obstacles 
in the slightly smaller ball concentric to the confinement ball. It may be known that one can give a simpler proof of Theorem 3 
using this fact in two dimensional case. However, since our proof is also applicable to two dimensional case, 
we have included it to Theorem 3. 

\vspace{11pt}
The outline of the article is as follows. Firstly we shall prove Theorem 2, which also implies the upper bounds 
of Theorem 1. Since this part is very similar to the one particle case, we only give the outline of the proof. 
Once we have shown Theorem 2, it suffices for the lower bounds to show that both $|W_t^C|$ and $|\Hat{W}_t^C|$ 
have the same volume as the confinement ball. We prove it considering exponential moments of $|W_t^C|$ as in \cite{F1}. 
Next, we show the upper bound of Theorem 3 following the argument to prove Proposition 3.2.7 in \cite{Szn98}. 
The main difficulty in our case is that the confinement ball need not be completely clear so that the process may 
avoid some parts of the ball for a long time. To get over this point, we consider the covering time of 
a nice set approximating the confinement ball instead. 
Finally, the lower bound of Theorem 3 follows from an exponential estimate for the Wiener sausage
which is due to van den Berg and T\'{o}th \cite{vdBT91}. 

\section{Proof of Theorem 2}
In this section, we shall give the proof of Theorem 2. 
As pointed out in the introduction, it is very similar to the one particle case. 
Let us start by introducing the Brownian scaling with scale $\e=t^{1/(d+2)}$. Under this scaling, we use the notation 
$\T=(-[t],[t])^d$ and 
\begin{equation*}
 v(\tau)=\exp \left\{-\int_{0}^{\tau}V_{\e}(Z_s, \omega_1)ds \right\}
 1_{\{H_{S_{\e}(\omega_2)}>\tau\}}
\end{equation*}
where $\tau=t \e^2$, $V_{\e}(x,\omega)=\e^{-2}\sum_i W((x-x_i)/\e)$ and $S_{\e}(\omega)=\bigcup_i (x_i+\e K)$. 
We further introduce the notation for the scaled version of the annealed path measure  
\begin{equation*}
 \Q^{\mu, \nu}_{t, \e}=\frac{1}{\S_t^{\mu,\nu}} v(\tau) \Hat{v}(\tau)
 {\P^1_{\mu \e^{-d}} \otimes \P^2_{\nu \e^{-d}} \otimes P_0 \otimes \Hat{P}_0} 
\end{equation*}
to simplify the presentation. 
Then, as in \cite{F1}, Theorem 2 follows once we have shown
\begin{equation*}
 \Q^{\mu, \nu}_{t, \e} 
 \left(T_{B(\omega_1,\omega_2)}\wedge\Hat{T}_{B(\omega_1,\omega_2)}>\tau \right) \to 1\label{9}
\end{equation*}
as $t \to \infty$. Since we have 
\begin{equation*}
 \lim_{t \to \infty}\frac{1}{t}\log{\Q^{\mu, \nu}_{t, \e}} 
 \left( T_{\T}\le \tau\textrm{ or }\Hat{T}_{\T}\le \tau \right) < 0\label{10}
\end{equation*}
from a standard estimates on Brownian motion, we restrict our consideration on 
$\{T_{\T}>\tau, \Hat{T}_{\T}>\tau \}$ in the sequel. Let us introduce the open set 
\begin{equation*}
 \mathscr{U}(\omega_1,\omega_2)=(\T \cap \O(\omega_1,\omega_2))\setminus \clD(\omega_1,\omega_2)\label{11}
\end{equation*}
and take the same parameters $\alpha_0$, $\beta$, $\rho$ and $\kappa$ as in \cite{F1}.  
Then, we have following constraint on this set. 
\begin{prp}\label{Proposition4}
Pick $\chi \in (0,1)$ such that
\begin{equation*}
 \chi >\max\left(\beta+\alpha_0,1-\left(\frac{\kappa}{d}-\alpha_0\right),
 1-\frac{\rho}{d}\right)\label{12}
\end{equation*}
and let
\begin{gather*}
 \alpha_1<\min(d(1-\chi),1),\label{13}\\
 C_1(a,d,\mu+\nu)>2(1+\gamma(a,d,\mu+\nu)),\label{14}
\end{gather*}
where $a=\inf\{u>0\,;\, {\rm supp}W \cup K \subset B(0,u)\}$ and $\gamma$ is a constant. 
Then we have 
\begin{equation*}
 \begin{split}
  &\Q^{\mu, \nu}_{t, \e}
  \left( T_{\T}>\tau, \Hat{T}_{\T}>\tau,
  (\mu+\nu)|\mathscr{U}|+ 2\lambda(\mathscr{U})>2^{\frac{d}{d+2}}c(d,\mu+\nu)+C_1\e^{\alpha_1}\right)\\
  &\le \exp\left\{-(1+\gamma)t^{\frac{d-\alpha_1}{d+2}}\right\}\label{15}
 \end{split}
\end{equation*}
as $t \to \infty$. 
\end{prp}
\vspace{11pt}\noindent
The proof of this proposition is essentially the same as Proposition 1 in \cite{F1}. 
We have to change only two parts. The first is the value of $\gamma(a,d,\mu+\nu)$ which comes from 
the lower bound on the normalizing constant. See \eqref{25} in section 3 for this. 
The second is that we have the squared semigroup 
\begin{equation*}
 E_0 \otimes \Hat{E}_0 \left[ v(\tau) \Hat{v}(\tau) \,;\,T_{\T}>\tau, \Hat{T}_{\T}>\tau\right]
 = E_0 \left[ v(\tau) \,;\,T_{\T}>\tau \right]^2\label{16}
\end{equation*}
for fixed $(\omega_1,\omega_2)$. As a result, we have another variational problem 
\begin{equation*}
 \inf_{U:{\rm open}}\set{(\mu+\nu)|U|+2\lambda(U)}\label{17}
\end{equation*}
which achieves minimum value $2c(d,(\mu+\nu)/2)=2^{d/(d+2)}c(d,\mu+\nu)$ at $U=B(0,R_0(d,(\mu+\nu)/2))$. 

Once this Proposition has been proved, the rest of the proof is just the same as in \cite{F1}. 
Indeed, we can construct confinement ball $B_l$ (whose radius is $2^{1/(d+2)}R_0+l$) with the help of the reinforcement 
of Faber-Krahn's inequality and also can prove 
\begin{equation*}
 \begin{split}
  &\Q_{t,\e}^{\mu,\nu}\left(T_{B_l(\omega_1,\omega_2)} \le \tau \textrm{ or } \Hat{T}_{B_l(\omega_1,\omega_2)} \le \tau\right)\\
  \le &\, 2\Q_{t,\e}^{\mu,\nu}\left( T_{B_l(\omega_1,\omega_2)} \le \tau \right)\\
  \le &\, \exp\left\{-c_2l t^{\frac{\alpha_3}{d+2}}\right\} \label{18}
 \end{split}
\end{equation*}
using the obvious version of Proposition 3 in \cite{F1}. 
Taking $l=t^{-\alpha_4/(d+2)}$($\alpha_4 < \alpha_3$), this implies Theorem 2.

\section{Lower estimate of Theorem 1}
In this section, we are going to show the lower estimate of Theorem 1. 
Let $u(t)$ denote the `killing term' 
\begin{equation*}
 \exp \left\{-\int_0^t V(Z_s,\omega_1)ds-\int_0^t V(\Hat{Z}_s,\omega_1)ds \right\}
 1_{\{H_{S(\omega_2)}>t, \Hat{H}_{S(\omega_2)}>t\}}\label{19}
\end{equation*}
to simplify the notation. 
The key ingredient is following asymptotic estimate. 
\begin{lmm}\label{Lemma 1}
For $0 \le \lambda < \mu+\nu$ and nonpolar compact set $C \subset \R^d$,
 \begin{equation}
  \log \E^1_{\mu} \otimes \E^2_{\nu} \otimes E_0 \otimes \Hat{E}_0\left[u(t)\exp \left\{ -\lambda|W_t^C| \right\}\right] 
  \sim -2c\left(d,\frac{\mu+\nu+\lambda}{2} \right) t^{\frac{d}{d+2}}.\label{20}
 \end{equation}
\end{lmm}
\vspace{11pt}\noindent
Indeed, this lemma and Chebyshev's inequality shows 
\begin{equation*}
 \begin{split}
  &\Q_t^{\mu,\nu}\left(t^{-\frac{d}{d+2}}|W_t^C| \le m \right)\\
  \le &\exp \left\{ \lambda\left(m-\frac{2c\left(d,\frac{\mu+\nu+\lambda}{2}\right)-2c\left(d,\frac{\mu+\nu}{2}\right)}{\lambda}\right)
  t^{\frac{d}{d+2}}(1+o(1))\right\} \label{21}
 \end{split}
\end{equation*}
and consequently, it follows that 
\begin{equation*}
 \lim_{t \to \infty}\Q_t^{\mu,\nu}\left(t^{-\frac{d}{d+2}}|W_t^C|\le
 2^{d/(d+2)} \omega_d R_0^d - \eta \right)=0\label{22}
\end{equation*}
for any $\eta >0$. Here we have used (12) of \cite{F1}:
\begin{equation*} 
 \frac{\partial}{\partial \nu}c(d,\nu)=\omega_d R_0(d,\nu)^d.
\end{equation*}

\vspace{11pt}\noindent
\noindent{\it Proof of Lemma 1.} First of all, note that 
\begin{equation*}
 \begin{split}
  &\E^1_{\mu} \otimes \E^2_{\nu} \otimes E_0 \otimes \Hat{E}_0\left[u(t)\exp \set{-\lambda |W_t^C|}\right]\\
  =\,&\E^1_{\mu \e^{-d}} \otimes \E^2_{\nu \e^{-d}} \otimes \E^3_{\lambda \e^{-d}} \otimes E_0 \otimes \Hat{E}_0
   \left[v(\tau)\Hat{v}(\tau)\,;\, H_{\Tilde{S}_{\e}(\omega_3)}>\tau \right]\label{23}
 \end{split}
\end{equation*}
where $\Tilde{S}_{\e}(\omega)=\bigcup_i (x_i-\e C)$ for $\omega=\sum_i \delta_{x_i}$. 
Let ${\bold P}_{\epsilon}$ denote $\P^1_{\mu \e^{-d}} \otimes \P^2_{\nu \e^{-d}} \otimes \P^3_{\lambda \e^{-d}}$ 
and $\bold E_{\epsilon}$ the corresponding expectation for simplicity.
To show the lower bound, we consider the specific event 
\begin{equation*}
 \begin{split}
  A_R&=\set{(\omega_1+\omega_2+\omega_3)(B(0,R+a\e))=0,T_{B(0,R)}>\tau,\Hat{T}_{B(0,R)}>\tau}\\ 
     &\subset \set{v(\tau)=1, \Hat{v}(\tau)=1, H_{\Tilde{S}_{\e}(\omega_3)}>\tau}\label{24}
 \end{split}
\end{equation*}
where $a=\inf\{u>0\,;\, {\rm supp}W \cup K \cup C \subset B(0,u)\}$. 
Then, setting $R=R_0(d,(\mu+\nu+\lambda)/2)$ and using well known eigenfunction expansion, we have 
\begin{equation}
 \begin{split}
  &{\bold P}_{\epsilon}\otimes P_0 \otimes \Hat{P}_0(A_R)\\
  \ge &\, {\rm const}(d)\exp \set{-2\frac{\lambda_d}{R^2}\tau-(\mu+\nu+\lambda)\omega_d(R+a\e)^d \tau}\\
  \ge &\, {\rm const}(d)\exp \set{-2c\left(d,\frac{\mu+\nu+\lambda}{2}\right)\tau-\gamma\tau^{\frac{d-1}{d}}}\label{25}
 \end{split}
\end{equation}
for some constant $\gamma(a,d,\mu+\nu+\lambda)$ and the lower bound of \eqref{20} follows. 
To prove the upper bound, we use the `method of enlargement of obstacles'. 
See section 3.1 of \cite{F1} for the notation and results. 
From now on, we fix the admissible collection of parameters
\begin{equation*}
 \alpha,\beta,\gamma,\delta,L,\rho,\kappa\label{26}
\end{equation*}
and pick 
\begin{gather*}
 M=4c\left(d,\frac{\mu+\nu+\lambda}{2}\right),\label{27}\\
 0<r<r_0(M),\label{28}\\
 R \in \N \textrm{ with } c_3(d)\left[ \frac{R}{4r} \right] \in [\log t, \log t+1),\label{29}\\
 n_0 \in \N  \textrm{ with } \frac{\mu+\nu+\lambda}{2}n_0 r^d \in [M, M+1).\label{30}
\end{gather*}
Here $c_3(d)$ is the constant used in \cite{F1}. 
Using these parameters, we set
\begin{equation*}
\begin{split}
 \D_1=\D_{\e}(\omega_1,\omega_2,\omega_3), \B_1=\B_{\e}(\omega_1,\omega_2,\omega_3), \\
 \A_1=\A_{\e}(\omega_1,\omega_2,\omega_3), \O_1=\O_{\e}(\omega_1,\omega_2,\omega_3), \label{31}
\end{split}
\end{equation*}
and
\begin{equation*}
\begin{split}
 \D_2=\D_{\e}(\omega_1,\omega_2), \B_2=\B_{\e}(\omega_1,\omega_2), \\
 \A_2=\A_{\e}(\omega_1,\omega_2), \O_2=\O_{\e}(\omega_1,\omega_2). \label{32}
\end{split}
\end{equation*}
Now let us define the {\textit{essential part}} by
\begin{equation*}
 E=\set{\lambda_{\omega_1,\omega_2,\omega_3}^{\e}(\T) \le M, |\A_2 \cap 2\T| \le n_0, T_{\T}>\tau, \Tilde{T}_{\T}>\tau}.\label{33}
\end{equation*}
It can easily be seen that 
\begin{equation}
 \D_1 \supset \D_2, \D_1\cup\B_1 \supset \D_2\cup\B_2, \A_1 \subset \A_2, \O_1 \subset \O_2\label{34}
\end{equation}
from the definition of these sets. Therefore, if $(\omega_1,\omega_2,\omega_3) \in E$ then we have 
$\lambda_{\omega_1,\omega_2}^{\e}(\T) \le M$ and $|\A_1 \cap 2\T| \le n_0$. 
We can also show that $E$ is essential, namely 
\begin{equation*}
 \limsup_{t \to \infty} \frac{1}{\tau} \log{\bold E}_{\epsilon}\otimes E_0 \otimes \Hat{E}_0
 \left[v(\tau)\Hat{v}(\tau)\,;\, H_{\Tilde{S}_{\e}(\omega_3)}>\tau, E^c \right] \le -M, \label{35}
\end{equation*}
by the same argument as to show Lemma 4.5.5 of \cite{Szn98}. 
For $(\omega_1,\omega_2,\omega_3) \in E$ we set 
\begin{equation*}
 \U_i=(\T \cap \O_i)\setminus \D_i,\V_i=(\T \cap \O_i) \setminus (\D_i \cup \B_i) \quad (i=1,2)\label{36}
\end{equation*}
so that $\U_1 \subset \U_2$ and $\V_1 \subset \V_2$ from \eqref{34} and $(\omega_1+\omega_2+\omega_3)(\V_1)=(\omega_1+\omega_2)(\V_2)=0$. 
Moreover, it follows from the volume control of \cite{F1} that
\begin{equation*}
 |\U_i| \le |\V_i| + |\T \cap \O_i| \e^{\kappa} \le |\V_i| + (2R+1)^d n_0 \e^{\kappa}.\label{37}
\end{equation*}
Now let us introduce the covering ${\mathcal G}_t$ of $E$ made of the events 
\begin{equation*}
 G_{U_1,V_1,U_2,V_2}=\set{\U_i=U_i, \V_i=V_i,i=1,2}\label{38}
\end{equation*}
which intersect with $E$. Then the cardinality of ${\mathcal G}_t$ is of order $\exp\set{o(\tau)}$ like 
(4.5.78) of \cite{Szn98}. Therefore the proof of the upper bound is reduced to `pointwise estimate', 
i.e.\ the estimate on each $G_{U_1,V_1,U_2,V_2}$: 
\begin{equation*}
 \begin{split}
  &{\bold E}_{\epsilon}\otimes E_0 \otimes \Hat{E}_0\left[v(\tau)\Hat{v}(\tau)\,;\,
  H_{\Tilde{S}_{\e}(\omega_3)}>\tau, G_{U_1,V_1,U_2,V_2}\cap E \right]\\
  \le \,& c(d)^2\left(1+(M\tau)^{\frac{d}{2}} \right)^2 {\bold E}_{\epsilon}
  \Bigl[\exp\set{-\left(\lambda^{\e}_{\omega_1,\omega_2,\omega_3}(\T)\wedge M+\lambda^{\e}_{\omega_1,\omega_2}(\T)\wedge M\right)\tau}\,;\\
  &\hspace{308pt}G_{U_1,V_1,U_2,V_2} \Bigr]\\
  \le \,& {\bold E}_{\epsilon}\Bigl[\exp\set{-\left(\lambda^{\e}_{\omega_1,\omega_2,\omega_3}(\T \cap \O_1)\wedge M
  +\lambda^{\e}_{\omega_1,\omega_2}(\T \cap \O_2)\wedge M+o(1)\right)\tau}\,;\\
  &\hspace{297pt}G_{U_1,V_1,U_2,V_2} \Bigr]\\
  \le \,& {\bold E}_{\epsilon}\Bigl[\exp\set{-\left(\lambda^{\e}_{\omega_1,\omega_2,\omega_3}(\U_1)\wedge M
  +\lambda^{\e}_{\omega_1,\omega_2}(\U_2)\wedge M+o(1)\right)\tau}\,;\,G_{U_1,V_1,U_2,V_2} \Bigr]\\
  \le \,& \exp\set{-\left(\lambda(U_1)\wedge M+\lambda(U_2)\wedge M+o(1)\right)\tau}
  {\bold P}_{\epsilon}\left(\omega_3(V_1)=0,\omega_1+\omega_2(V_2)=0\right)\\
  \le \,& \exp\set{-\left((\lambda(U_1)+\lambda(U_2)+\lambda|U_1|+(\mu+\nu)|U_2|+o(1))\wedge M\right)\tau}.\label{39}
 \end{split}
\end{equation*}
Here we have used (3.1.9) of \cite{Szn98} in the second line, spectral control III of \cite{F1} in the third line, 
spectral control I of \cite{F1} in the fourth line and $V_1 \subset V_2$ in the fifth line. 
The upper bound on the last line comes from 
\begin{equation*}
 \begin{split}
  &\inf_{U_1 \subset U_2:{\rm open}}\set{\lambda(U_1)+\lambda(U_2)+\lambda|U_1|+(\mu+\nu)|U_2|}\\
  &\hspace{2.2pt}=\inf_{R_1 \le R_2}
  \set{\frac{\lambda_d}{R_1^2}+\frac{\lambda_d}{R_2^2}+\lambda\omega_d R_1^d+(\mu+\nu)\omega_d R_2^d}.\label{40}
 \end{split}
\end{equation*}
A little calculus shows that this variational problem attains the infimum $2c(d,(\mu+\nu+\lambda)/2)$ 
at $r=R=R_0(d,(\mu+\nu+\lambda)/2)$ and the proof of Lemma 1 is completed. 
\hfill $\square$

\section{Estimates on the covering time}
We shall prove Theorem 3 in this section. 
Throughout this section, we adopt usual scaling with $\e= t^{1/(d+2)}$ and only consider 
$(\omega_1,\omega_2)$ for which the confinement property holds. 
Moreover, we use the method of enlargement of obstacles with the same parameters as in \cite{F1}. 
Under these settings, we let $B$ denote the scaled confinement ball 
$B(\omega_1,\omega_2)$ in Theorem 1 of \cite{F1}, $\lambda_{\omega_1,\omega_2}^{\e}$ the principal eigenvalue 
of $-1/2\Delta+V_{\e}(\,\cdot\,,\omega_1) $ on $H_0^1(B \setminus S_{\e}(\omega_2))$ and $\phi_{\omega_1,\omega_2}^{\e}$ the 
corresponding $L^2$-normalized positive eigenfunction. Finally, we introduce the scaled path measure
\begin{equation*}
 Q_{t,\e}^{\mu,\nu}=\frac{1}{S_t^{\mu,\nu}}v(\tau)\P^1_{\mu \e^{-d}} \otimes \P^2_{\nu \e^{-d}} \otimes P_0 
\end{equation*}
as in section 2. 

Let us start by recalling the asymptotics for the normalizing constant:
\begin{equation}
 S_t^{\mu,\nu}=\exp\set{-c(d,\mu+\nu)t^{\frac{d}{d+2}}+o\left(t^{\frac{d}{d+2}} \right)} \quad (t \to \infty),\label{41}
\end{equation}
which we will use in the sequel (see for instance (3) in \cite{F1}). 
Now, we shall prove two lemmas to approximate $B \setminus S_{\e}$ 
by nice sets. The first is the level set of the eigenfunction $\phi_{\omega_1,\omega_2}^{\e}$. 

\begin{lmm}
For any $\e_1>0$, there exists $\Omega_{t}(\e_1) \subset \Omega^2$ such that 
 \begin{gather}
  \lambda_{\omega_1,\omega_2}^{\e} \le \lambda(B(0,R_0(d,\mu+\nu)))+\e_1 \textrm{ \rm{for} } (\omega_1,\omega_2) 
  \in \Omega_t(\e_1),\label{42}\\
  \lim_{t \to \infty}Q_{t,\e}^{\mu,\nu} \left( \Omega_{t}(\e_1) \right) = 1.\label{44}    
 \end{gather}
Moreover, when $C_2$ is large enough depending only on the dimension and $\mu+\nu$, we have 
 \begin{equation}
  |\{\phi_{\omega_1,\omega_2}^{\e}>s\}| \ge |B(0,R_0(d,\mu+\nu))|(1-C_2(\e_1+s))\label{43}
 \end{equation}
for any $(\omega_1,\omega_2) \in \Omega_t(\e_1)$ and $s>0$. 
\end{lmm}
\begin{proof}
By the confinement property, we can restrict our consideration on $\{T_{B} > \tau \}$. 
Furthermore, we can admit another restriction $\lambda_{\omega_1,\omega_2}^{\e} \le 2c(d,\mu+\nu)$, since
\begin{equation*}
 \begin{split}
  &\E^1_{\mu \e^{-d}} \otimes \E^2_{\nu \e^{-d}} \otimes E_0 
  \left[v(\tau) \,;\, \lambda_{\omega_1,\omega_2}^{\e} > 2c(d,\mu+\nu) \right]\\
  \le \,&\E^1_{\mu \e^{-d}} \otimes \E^2_{\nu \e^{-d}}\left[c(d)\left(1+(\lambda_{\omega_1,\omega_2}^{\e}\tau)^{d/2}\right) \right.
  \exp\set{-\lambda_{\omega_1,\omega_2}^{\e}\tau} \,;\\
  &\hspace{187pt}\lambda_{\omega_1,\omega_2}^{\e} > 2c(d,\mu+\nu) \Bigr]\\
  \le \,& c'(d) \exp\set{-\frac{3}{2}c(d,\mu+\nu)\tau}\\
  = \,&o(S_t^{\mu,\nu}) \quad (t \to \infty).\label{45}
 \end{split}
\end{equation*}
Here we have used (3.1.9) of \cite{Szn98} in the first line, $\sup_{\lambda>0}\{(1+\lambda^{d/2})\exp\{-\lambda/4\}\}<\infty$ 
in the second line and \eqref{41} in the last line. 
On the other hand, it follows from the method of enlargement of obstacles that
\begin{equation*}
 \begin{split}
  \lambda_{\omega_1,\omega_2}^{\e}\wedge 2c(d,\mu+\nu) &\ge \lambda_{\omega_1,\omega_2}^{\e}(B \setminus \clD)\wedge 2c(d,\mu+\nu)-\e^{\rho}\\
  &\ge \lambda_{\omega_1,\omega_2}^{\e}(\T \setminus \clD)\wedge 2c(d,\mu+\nu)-\e^{\rho}\\
  &\ge \lambda_{\omega_1,\omega_2}^{\e}(\U)\wedge 2c(d,\mu+\nu)-2\e^{\rho}\label{46}
 \end{split}
\end{equation*}
with $\U=(\T\cap \O) \setminus \clD$. Therefore, for any $\e_1>0$ we have
\begin{equation*}
 \begin{split}
  &\E^1_{\mu \e^{-d}} \otimes \E^2_{\nu \e^{-d}} \otimes E_0 
  \left[v(\tau) \,;\, \lambda(B(0,R_0))+\e_1 < \lambda_{\omega_1,\omega_2}^{\e} \le 2c(d,\mu+\nu) \right]\\
  \le \,&\E^1_{\mu \e^{-d}} \otimes \E^2_{\nu \e^{-d}}\biggl[c(d)\left(1+(2c(d,\mu+\nu)\tau)^{d/2}\right)\\ 
  &\hspace{77pt} \exp\set{-((\lambda(B(0,R_0))+\e_1)\vee(\lambda(\U)-2\e^{\rho})) \tau} \biggr]\\
  \le \,&\sum_{U,V}\exp\set{-((\lambda(B(0,R_0))+\e_1)\vee(\lambda(U)-2\e^{\rho}))\tau+o(\tau)}\P_{\mu \e^{-d}} \otimes \P_{\nu \e^{-d}}(G_{U,V})\\
  \le \,&\exp\set{-\inf_{U:{\rm open}}\set{(\lambda(B(0,R_0))+\e_1)\vee(\lambda(U)-2\e^{\rho})+(\mu+\nu)|U|}\tau+o(\tau)}\label{47}
 \end{split}
\end{equation*}
as in (4.5.81) of \cite{Szn98}. The infimum in the last line turns out to be larger than 
\begin{equation*}
 c(d,\mu+\nu)+C_3(d,\mu+\nu)\e_1^2\label{48} 
\end{equation*}
after some calculation and this shows the existence of $\Omega_t(\e_1)$ with the properties \eqref{42} and \eqref{44}. 

Next, we shall prove that \eqref{43} holds on this $\Omega_t(\e_1)$. Let us start by two obvious estimates 
\begin{equation*}
 \|(\phi_{\omega_1,\omega_2}^{\e}-s)_{+}\|_2 
 \ge \|\phi_{\omega_1,\omega_2}^{\e}\|_2 - \|\phi_{\omega_1,\omega_2}^{\e}\wedge s\|_2 \ge 1-s|B \setminus S_{\e}|^{1/2}\label{49}
\end{equation*}
and
\begin{equation*}
 \begin{split}
  &\frac{1}{2}\int |\nabla (\phi_{\omega_1,\omega_2}^{\e}(x)-s)_{+}|^2 dx + \int V(x,\omega_1) (\phi_{\omega_1,\omega_2}^{\e}(x)-s)_{+}^2 dx\\
  \le\,& \frac{1}{2}\int |\nabla\phi_{\omega_1,\omega_2}^{\e}(x)|^2 dx + \int V(x,\omega_1) \phi_{\omega_1,\omega_2}^{\e}(x)^2 dx\\
  =\, &\lambda_{\omega_1,\omega_2}^{\e}.\label{50}
 \end{split}
\end{equation*}
Combining these estimates we find 
\begin{equation*}
 \begin{split}
  &\lambda(\{\phi_{\omega_1,\omega_2}^{\e}>s\})\\
  =\,&\inf_{\varphi \in H_0^1(\{\phi_{\omega_1,\omega_2}^{\e}>s\}),\, \|\varphi\|_2=1}
  \set{\frac{1}{2}\int |\nabla\varphi(x)|^2 dx + \int V(x,\omega_1) \varphi(x)^2 dx }\\
  \le\,& \lambda_{\omega_1,\omega_2}^{\e}\left(1-s|B \setminus S_{\e}|^{\frac{1}{2}} \right)^{-2}.\label{51}
 \end{split}
\end{equation*}
On the other hand, we also have a converse estimate 
\begin{equation*}
 \lambda(\{\phi_{\omega_1,\omega_2}^{\e}>s\}) 
 \ge \lambda_{d}\left(\frac{\omega_d}{|\{\phi_{\omega_1,\omega_2}^{\e}>s\}|} \right)^{\frac{2}{d}}\label{52}
\end{equation*}
from Faber-Krahn's inequality (see e.g.\ \cite{Ber03}). 
Therefore it follows for $(\omega_1,\omega_2) \in \Omega_t(\e_1)$ that 
\begin{equation*}
 \begin{split}
  |\{\phi_{\omega_1,\omega_2}^{\e}>s\}| \ge \,&\omega_d\left(\frac{\lambda_d}{\lambda(\{\phi_{\omega_1,\omega_2}^{\e}>s\})} \right)^{\frac{d}{2}}\\
  \ge \,&\omega_d\left(\frac{\lambda_d}{\lambda_{\omega_1,\omega_2}^{\e}} \right)^{\frac{d}{2}}\left(1-s|B \setminus S_{\e}|^{\frac{1}{2}} \right)^{d}\\
  \ge \,&\omega_d\left(\frac{\lambda_d}{\lambda(B(0,R_0))+\e_1} \right)^{\frac{d}{2}}\left(1-s|B \setminus S_{\e}|^{\frac{1}{2}} \right)^{d}
  \label{53} 
 \end{split}
\end{equation*}
and our claim \eqref{43} follows.
\end{proof}
The second is the set of points in $B$ which keep certain distance from $\partial B$ and obstacles. 
\begin{lmm}
If we define the set
\begin{equation*}
 \mathscr{W}=\set{x \in B \,;\, {\rm dist}(x,\partial B \cup \clD \cup \overline{\B})>3a\e }.\label{56}
\end{equation*}
for $a>0$, then we have 
\begin{equation*}
 |\mathscr{W}| \ge \omega_d R_0^d - C_4(d,\mu+\nu)\e^{\frac{\alpha \wedge \alpha_1}{2} \wedge \kappa} \label{57}
\end{equation*}
for large enough $t$. Here $\alpha_1$ is the same constant as in \cite{F1}. 
\end{lmm}
\begin{proof}
Firstly,  we have following estimate on slightly larger neighborhood of $\partial B$:
\begin{equation}
 \left|\set{x \in B \,;\, {\rm dist}(x,\partial B) \le \e^{\alpha} }\right|
 \le C_4(d,\mu+\nu)\e^{\alpha}\label{58}
\end{equation}
since $B$ has the radius in $[R_0,R_0+\kappa_1\e^{\kappa_2}]$. Next, we shall deal with the neighborhood of $\clD$. 
From Proposition 2 of \cite{F1}, we have 
\begin{equation*}
  |B \cap \clD | \le \left|B \setminus \U \right| \le  c_6(d,\mu+\nu) \e^{\frac{\alpha_1}{2}}\label{59}
\end{equation*}
where $c_6$ is the constant used in \cite{F1}. Now, let us recall that the density set $\D$ consists of boxes 
with side length $L^{n_{\gamma}} \in [\e^{\gamma},L\e^{\gamma})$ ($\alpha<\gamma < 1$, 
see section 3.1 of \cite{F1}). If we denote by $\D'$ the consisting boxes of $\D$ which intersects with $\partial B$, we have 
\begin{equation*}
 \set{x \in B \,;\, {\rm dist}(x,\D') \le 3a\e} \subset \set{x \in B \,;\, {\rm dist}(x,\partial B) \le \e^{\alpha} }\label{60}
\end{equation*}
for large enough $t$ since $\alpha<\gamma$. On the other hand, we know
\begin{equation}
 \begin{split}
  &\left| \set{x \in B \,;\, {\rm dist}(x,\overline{\D\setminus\D'}) \le 3a\e} \right|\\ 
  \le \,&\sum |3a\e \textrm{-neighborhood of each box of }\overline{\D\setminus\D'} |\\
  \le\, &\frac{(\e^{\gamma}+6a\e)^d}{\e^{d\gamma}}|\clD|\\
  \le\, &2|\clD|.\label{60b}
 \end{split}
\end{equation}
for large enough $t$ since $\gamma<1$. From \eqref{58}--\eqref{60b}, we get 
\begin{equation}
 \left|\set{x \in B \,;\, {\rm dist}(x,\partial B \cup \clD ) \le 3a\e }\right|
 \le C_4(d,\mu+\nu)\e^{\alpha \wedge \frac{\alpha_1}{2}}\label{61}
\end{equation}
making $C_4$ larger if necessary. 
Finally, since we know from the volume control of \cite{F1} that 
\begin{equation*}
 |B \cap \overline{\B} | \le \bigcup_{q \in [-2R_0-1,2R_0+1]^d \cap \Z^d} |(q+[0,1)^d) \cap \overline{\B} | 
 \le (4R_0+3)^d \e^{\kappa}\label{62}
\end{equation*}
and $\B$ consists of boxes with side length $L^{n_{\beta}} \in [\e^{\beta},L\e^{\beta})$ ($\alpha < \beta < 1$, 
see section 3.1 of \cite{F1}), we can show 
\begin{equation}
 \left|\set{x \in B \,;\, {\rm dist}(x,\partial B \cup \overline{\B} ) \le 3a\e }\right|
 \le C_4(d,\mu+\nu)\e^{\alpha \wedge \kappa}\label{63}
\end{equation}
as before, making $C_4$ larger if necessary. 

Combining \eqref{58}, \eqref{61} and \eqref{63}, the proof of Lemma 3 is completed. 
\end{proof}

{\flushleft{\it Proof of Theorem 3.}} We shall prove \eqref{7} first. Since $C$ has non-empty interior, 
we can assume $C=\overline{B}(0,r)$ for some $r>0$. Let us introduce the positive constant
\begin{equation*}
 a=\inf\set{u>2r\,;\, {\rm supp}W \cup K \subset \overline{B}(0,u)} \label{64}
\end{equation*}
and define $\mathscr{W}$ accordingly. 
From Lemma 2 and Lemma 3, it suffices to show $W_{t^{\sigma}\e^2}^{\e C}$ covers $\mathscr{W} \cap \{\phi_{\omega_1,\omega_2}^{\e}>C_5 s\}$ 
on $\Omega_t(\e_1,s)$ for any $\e_1>0$, $s>0$ and some appropriate constant $C_5(r,d,\mu+\nu)>1$. 
To this end, we introduce a covering of $\mathscr{W} \cap \{\phi_{\omega_1,\omega_2}^{\e}> s\}$ first. 
Let $B_q$ $ (q \in \Z^d)$ be the closed ball $\overline{B}(r\e/(2\sqrt{d})q,r\e/2)$ and 
\begin{equation*}
 {\mathcal I}(\omega_1,\omega_2) = \set{q \in \Z^d \,;\, B_q \cap \mathscr{W} \cap \{\phi_{\omega_1,\omega_2}^{\e}>s\} \neq \emptyset}\label{65}
\end{equation*}
so that
\begin{equation*}
 \bigcup_{q \in {\mathcal I}(\omega_1,\omega_2)} B_q \supset \mathscr{W} \cap \{\phi_{\omega_1,\omega_2}^{\e}>s\}.\label{66}
\end{equation*}
The cardinality of ${\mathcal I}(\omega_1,\omega_2)$ is uniformly bounded by some polynomial $p_1(t)$ which depends only on 
$d$ and $\mu+\nu$ since we always have 
\begin{equation*}
 {\mathcal I}(\omega_1,\omega_2) \subset \set{q \in \Z^d \,;\, B_q \cap \overline{B}(0,2R_0+1) \neq \emptyset}.\label{67}
\end{equation*}
Next we exclude $B_q$ which intersects with $\mathscr{W} \cap \{\phi_{\omega_1,\omega_2}^{\e}= s\}$. (We can prove 
$B_q \subset W_{t^{\sigma}\e^2}^{\e C}$ only when $B_q$ is included in $\{\phi_{\omega_1,\omega_2}^{\e}> s\}$.) 
Since there are no obstacles in $a\e$-neighborhood of $\clD \cup \overline{\B}$, $\phi^{\e}_{\omega_1,\omega_2}$ 
is the solution of the elliptic equation 
\begin{equation*}
 \left(\frac{1}{2}\Delta - \lambda^{\e}_{\omega_1,\omega_2}\right) \phi^{\e}_{\omega_1,\omega_2} =0 
 {\textrm{ on }} \set{x \in B \,;\, {\rm dist}(x, \partial B \cup \clD \cup \overline{\B})>a\e }.\label{68}
\end{equation*}
Moreover, by the definitions of $a$ and $\mathscr{W}$ we have 
\begin{equation*}
 \begin{split}
  \overline{B}(r\e/(2\sqrt{d})q,2r\e) &\subset \overline{B}(r\e/(2\sqrt{d})q,a\e)\\
  &\subset \set{x \in B \,;\, {\rm dist}(x,\partial B \cup \clD \cup \overline{\B})>a\e }\label{69}
 \end{split}
\end{equation*}
for such $q$. Therefore we can use the Harnack inequality (see Theorem 8.20 in \cite{GT01}) for $\phi^{\e}_{\omega_1,\omega_2}$ 
to get 
\begin{equation*}
 \begin{split}
  \sup_{B_q} \phi^{\e}_{\omega_1,\omega_2} 
  &\le \exp\set{{\rm const}(d)\left(\sqrt{d}+\sqrt{2\lambda^{\e}_{\omega_1,\omega_2}}\frac{r\e}{2}\right)}\inf_{B_q} \phi^{\e}_{\omega_1,\omega_2}\\
  &\le C_5 s \label{70}.
 \end{split}
\end{equation*}
with some constant $C_5(r,d,\mu+\nu)>1$. Here we have used the boundedness of $\lambda^{\e}_{\omega_1,\omega_2}$ 
in Lemma 1 and $B_q \cap \{\phi^{\e}_{\omega_1,\omega_2}=s\}\not= \emptyset$ in the second line. 
As a consequence, we have 
\begin{equation*}
 \mathscr{W} \cap \{\phi_{\omega_1,\omega_2}^{\e}>C_5s\}
 \subset \bigcup_{q \in {\mathcal J}(\omega_1,\omega_2)} B_q 
 \subset  \{\phi_{\omega_1,\omega_2}^{\e}>s\}\label{71}
\end{equation*}
where ${\mathcal J}(\omega_1,\omega_2)=\{q \in {\mathcal I}(\omega_1,\omega_2)\,;\, B_q \subset \{\phi_{\omega_1,\omega_2}^{\e}>s\}\}$. 
For a technical reason, we make a sequence $\{q_i(\omega_1,\omega_2)\}_{i=1}^{[p_1(t)]+1}$ of deterministic length $[p_1(t)]+1$ 
arranging all the points of ${\mathcal J}(\omega_1,\omega_2)$ redundantly. 

Now, we shall derive the upper bound on the probability 
\begin{equation}
 Q_{t,\e}^{\mu,\nu}\left( B_{q_i} \not\subset W_{t^{\sigma}\e^2}^{\e C} , \Omega_t(\e_1,s) \right).\label{72}
\end{equation}
Since $\{B_{q_i} \not\subset W_{t^{\sigma}\e^2}^{\e C}\}\subset\{H_{B_{q_i}}>t^{\sigma}\e^2\}$ by the definition of $B_q$, 
we can replace $\{B_{q_i} \not\subset W_{t^{\sigma}\e^2}^{\e C}\}$ by $\{H_{B_{q_i}}>t^{\sigma}\e^2\}$ in \eqref{72}. 
Hereafter, $\lambda_{\omega_1,\omega_2}^{\e,i}$ denotes the principal eigenvalue of $-1/2\Delta+V_{\e}(\,\cdot\,,\omega_1) $ 
on $H_0^1(B \setminus (S_{\e}(\omega_2) \cup B_{q_i}))$. 
Then, using Markov property and (3.1.9) in \cite{Szn98}, we find 
\begin{equation}
 \begin{split}
  &E_0\left[v(t\e^2) \,;\, T_B>t\e^2, H_{B_{q_i}}>t^{\sigma}\e^2\right]\\
  =\, &E_0\Bigl[v(t^{\sigma}\e^2)E_{Z_{t^{\sigma}\e^2}}\left[v(t\e^2-t^{\sigma}\e^2)\,;\,T_B>t\e^2-t^{\sigma}\e^2\right] \,;\\
  &\hspace{143pt}T_B>t^{\sigma}\e^2, H_{B_{q_i}}>t^{\sigma}\e^2\Bigr]\\
  \le \,&c(d)^2 \left(1+(\lambda_{\omega_1,\omega_2}^{\e,i}t\e^2)^{d/2}\right)^2 \\
  &\,\exp\set{-(\lambda_{\omega_1,\omega_2}^{\e,i}-\lambda_{\omega_1,\omega_2}^{\e})t^{\sigma}\e^2-\lambda_{\omega_1,\omega_2}^{\e}t\e^2}.\label{73}
 \end{split}
\end{equation}
Here we have used $\lambda_{\omega_1,\omega_2}^{\e,i} \ge \lambda_{\omega_1,\omega_2}^{\e}$ and 
$t^{\sigma}\e^2 \vee (t\e^2-t^{\sigma}\e^2) \le t\e^2$ in the last line. 
For the spectral shift $\lambda_{\omega_1,\omega_2}^{\e,i}-\lambda_{\omega_1,\omega_2}^{\e}$ in the last line, 
we have following lower bound. 
\begin{lmm}
When $\e_1>0$ and $\e>0$ are small enough, we have 
\begin{equation}
 \lambda_{\omega_1,\omega_2}^{\e,i}-\lambda_{\omega_1,\omega_2}^{\e} \ge C_6(r,d,\mu+\nu)s^2 h(\e)\label{74}
\end{equation}
for all $(\omega_1,\omega_2) \in \Omega_t(\e_1,s)$. Here $h$ is the function defined by 
\begin{equation*}
 \begin{split}
  h(\e)=\left\{
   \begin{array}{lr}
    \left(\log\frac{1}{\e}\right)^{-1} &(d=2),\\[8pt]
    \e^{d-2} &(d \ge 3).\label{75}
   \end{array}\right.
 \end{split}
\end{equation*}
\end{lmm}
\begin{proof}
From the Exercise 1) after Theorem 3.2.3 in \cite{Szn98}, the right hand side of \eqref{74} is larger than 
\begin{equation*}
 \left(1-\frac{\lambda_{\omega_1,\omega_2}^{\e}}{\mu_{\omega_1,\omega_2}^{\e}}\right)
 \inf_{B_{q_i}}(\phi_{\omega_1,\omega_2}^{\e})^2 {\rm cap}(B_{q_i}).\label{76}
\end{equation*}
Here $\mu_{\omega_1,\omega_2}^{\e}$ denotes the second smallest eigenvalue of $-1/2\Delta+V_{\e}(\,\cdot\,,\omega_1)$ on 
$H_0^1(B \setminus S_{\e}(\omega_2))$. If we denote by $\mu(U)$ the second smallest eigenvalue of $-1/2\Delta$ on $H_0^1(U)$, 
it easily follows from the Rayleigh-Ritz variational formula that $\mu(B(0,R_0+\kappa_1 {\e}^{\kappa_2})) \le \mu(B) \le \mu_{\omega_1,\omega_2}^{\e}$. 
($\kappa_1$ and $\kappa_2$ are the same constants as in Theorem 1 of \cite{F1}.) 
Therefore we find for any $(\omega_1,\omega_2) \in \Omega_t(\e_1)$ and small enough $\e$ that 
\begin{equation*}
 \lambda_{\omega_1,\omega_2}^{\e,i}-\lambda_{\omega_1,\omega_2}^{\e} 
 \ge \left(1-\frac{\lambda(B(0,R_0))+\e_1}{\mu(B(0,R_0+\kappa_1 {\e}^{\kappa_2}))}\right)s^2 C_7(d)h(r\e).\label{77}
\end{equation*}
This, together with the fact that 
\begin{equation*}
 \frac{\lambda(B(0,R_0))+\e_1}{\mu(B(0,R_0+\kappa_1 {\e}^{\kappa_2}))} \to \frac{\lambda(B(0,R_0))}{\mu(B(0,R_0))}<1 
 \quad \rm{as} \quad \e,\e_1 \to 0,\label{78}
\end{equation*}
completes the proof. 
\end{proof}
On the event $\{\lambda_{\omega_1,\omega_2}^{\e,i}\le 3c(d,\nu)\}$, the polynomial factor 
in the last line of \eqref{73} is uniformly bounded by $p_2(t)=c(d)^2 (1+(3c(d,\mu+\nu)t\e^2)^{d/2})^2$. 
Moreover, it follows from Lemma 2 of \cite{F1} that 
\begin{equation}
 \frac{1}{S_t^{\mu,\nu}}\E^1_{\mu \e^{-d}} \otimes \E^2_{\nu \e^{-d}} \left[\exp\set{-\lambda_{\omega_1,\omega_2}^{\e}t\e^2} \right]
 \le p_3(t)\label{79}
\end{equation}
for some polynomial $p_3(t)$ depending only on $d$ and $\mu+\nu$. 
Combining \eqref{73}, Lemma 4 and \eqref{79}, we have 
\begin{equation}
 \begin{split}
  &Q_{t,\e}^{\mu,\nu}\left( H_{B_{q_i}}>t^{\sigma}\e^2, \lambda_{\omega_1,\omega_2}^{\e,i}\le 3c(d,\nu) ,\Omega_t(\e_1,s) \right)\\
  \le &\, p_2(t)p_3(t)\exp\set{-{\rm const}(r,d,\mu+\nu)s^2 h(\e)t^{\sigma}\e^2},\label{80}
 \end{split}
\end{equation}
whose right hand side converges to zero faster than any polynomial provided that $\sigma>d/(d+2)$. 
The remaining part $\{\lambda_{\omega_1,\omega_2}^{\e,i}> 3c(d,\mu+\nu)\}$ is easier. Indeed, we have 
$\lambda(B(0,R_0))+\e_1 \le c(d,\mu+\nu)$ for $\e_1 \le |B(0,R_0)|$ and consequently
\begin{equation}
 \begin{split}
  &c(d)^2 \left(1+(\lambda_{\omega_1,\omega_2}^{\e,i}t\e^2)^{d/2}\right)^2 
  \exp\set{-(\lambda_{\omega_1,\omega_2}^{\e,i}-\lambda_{\omega_1,\omega_2}^{\e})t^{\sigma}\e^2-\lambda_{\omega_1,\omega_2}^{\e}t\e^2}\\
  \le&\, c(d)^2 \left(1+(\lambda_{\omega_1,\omega_2}^{\e,i}t\e^2)^{d/2}\right)^2 
  \exp\set{- \frac{2}{3}\lambda_{\omega_1,\omega_2}^{\e,i}t^{\sigma}\e^2-\lambda_{\omega_1,\omega_2}^{\e}t\e^2}\\
  \le&\, c'(d,\mu+\nu) \exp\set{- \frac{1}{3}\lambda_{\omega_1,\omega_2}^{\e,i}t^{\sigma}\e^2-\lambda_{\omega_1,\omega_2}^{\e}t\e^2}\\
  \le&\, c'(d,\mu+\nu) \exp\set{-c(d,\mu+\nu)t^{\sigma}\e^2-\lambda_{\omega_1,\omega_2}^{\e}t\e^2} \label{81}
 \end{split}
\end{equation}
on $\{\lambda_{\omega_1,\omega_2}^{\e,i}> 3c(d,\mu+\nu)\}\cap \Omega_t(\e_1,s)$. Here we have used 
\begin{equation*}
 \sup_{\lambda>3c(d,\mu+\nu),t>0}\set{\left(1+(\lambda t)^{d/2}\right)\exp\set{-\frac{1}{3}\lambda t^{\sigma}}}<\infty\label{82}
\end{equation*}
in the third line. Substituting \eqref{81} for \eqref{73} and using \eqref{79}, we find
\begin{equation}
 \begin{split}
  &Q_{t,\e}^{\mu,\nu}\left( H_{B_{q_i}}>t^{\sigma}\e^2, \lambda_{\omega_1,\omega_2}^{\e,i}> 3c(d,\nu) ,\Omega_t(\e_1,s) \right)\\
  \le&\, c'(d)p_3(t)\exp\set{-c(d,\mu+\nu)t^{\sigma}\e^2},\label{83}
 \end{split}
\end{equation}
whose right hand side converges to zero faster than any polynomial. 
Now that we have 
\begin{equation*}
 \begin{split}
  &Q_{t,\e}^{\mu,\nu}\left( \mathscr{W} \cap \{\phi_{\omega_1,\omega_2}^{\e}>s\} \not\subset W_{t^{\sigma}\e^2}^{\e C}, 
  \Omega_t(\e_1,s) \right)\\
  \le\, &Q_{t,\e}^{\mu,\nu}\left( \bigcup_{1 \le i \le[p_1(t)]+1} \set{H_{B_{q_i}}>t^{\sigma}\e^2}, \Omega_t(\e_1,s) \right)\\
  \le\, &\sum_{i=1}^{[p_1(t)]+1}Q_{t,\e}^{\mu,\nu}\left( H_{B_{q_i}}>t^{\sigma}\e^2, \Omega_t(\e_1,s) \right)\\
  \to &\, 0 \quad as \quad t \to \infty\label{84}
 \end{split}
\end{equation*}
from \eqref{80} and \eqref{83}, the proof of \eqref{7} is completed. 

Finally, we shall prove the lower estimate \eqref{8}. It is a consequence of following exponential estimate for the Wiener sausage 
\begin{equation}
 \lim_{t \to \infty} \frac{1}{t}\log E_0 \left[\exp\set{\lambda|W_t^C|}\right] = S(\lambda,r) \in (0,\infty)\label{85}
\end{equation}
which is due to van den Berg and T\'{o}th \cite{vdBT91}. 
Indeed, for fixed $\sigma<d/(d+2)$ and $\eta>0$, \eqref{85} yields a large deviation estimate
\begin{equation*}
 \begin{split}
  P_0\left( |W_{t^{\sigma}}^C|>\eta t^{\frac{d}{d+2}} \right) 
  \le\, &\exp\set{-\lambda\eta t^{\frac{d}{d+2}}+S(\lambda,r)t^{\sigma}+o\bigl(t^{\sigma}\bigr)}\\
  =\, &\exp\set{-2c(d,\mu+\nu)t^{\frac{d}{d+2}}+o\left(t^{\frac{d}{d+2}}\right)}\label{86}
 \end{split}
\end{equation*}
as $t \to \infty$, if we take $\lambda=2c(d,\mu+\nu)/\eta$. This shows \eqref{8}, since 
\begin{equation*}
 \begin{split}
  Q_t^{\mu,\nu}\left( |W_{t^{\sigma}}^C|>\eta t^{\frac{d}{d+2}} \right)
  \le\, &\frac{1}{S_t^{\mu,\nu}}P_0\left( |W_{t^{\sigma}}^C|>\eta t^{\frac{d}{d+2}} \right)\\
  \le\, &\exp\set{-c(d,\mu+\nu)t^{\frac{d}{d+2}}+o\left(t^{\frac{d}{d+2}}\right)}\label{87}
 \end{split}
\end{equation*}
in view of \eqref{41}. \hfill $\square$

{\scshape
\begin{flushright}
\begin{tabular}{l}
Division of Mathematics,\\ 
Graduate School of Science\\ 
Kyoto University, \\
Kyoto 606-8502, \\
Japan \\
{\upshape e-mail: fukusima@math.kyoto-u.ac.jp}\\
\end{tabular}
\end{flushright}
}

\end{document}